\documentclass[11pt]{article}
\usepackage{authblk}
\usepackage{graphicx}
\usepackage{amsfonts}
\usepackage{amsbsy}
\usepackage{amssymb}
\usepackage{amsmath}
\usepackage{makeidx}
\usepackage{tikz}
\usetikzlibrary{positioning}
\usepackage{amsmath,amsthm}
\usepackage{latexsym}
\usepackage{epsfig}
\usepackage{hyperref}
\usepackage{graphics}
\usepackage{float, floatflt, here}
\newtheorem{thm}{Theorem}[section]

\newtheorem{defn}[thm]{Definition}

\newtheorem{lemma}[thm]{Lemma}

\newtheorem{prop}[thm]{Proposition}
\newtheorem{rem}[thm]{Remark}

\newtheorem{cor}[thm]{Corollary}

\newtheorem{eg}[thm]{Example}
\newcommand{\Ker}{\operatorname{Ker}}
\title{Group action-stabilizer graph of group actions of a group on a set}
\date{{}}\author[1, 2]{Nepur Ranjan Hazarika \thanks{Corresponding author. Email address : nepurrhazarika@gmail.com}}
\author[1]{Kukil Kalpa Rajkhowa\thanks{Email address : kukilrajkhowa@yahoo.com}}
\small
\affil[1]{Department of Mathematics, Cotton University, Guwahati-781001, Assam, India}
\affil[2]{Department of Mathematics, North Gauhati College, College Nagar, Guwahati-781031, Assam, India}

\begin{document}
	\maketitle{}

	\begin{abstract}
		In this paper we introduce the group action-stabilizer graph $GAS(G)$ of a group $G$ on a set $X$. Its vertex set is the collection of all non-trivial group actions of $G$ on $X$, and two vertices $\phi$ and $\psi$ are adjacent if and only if the normal subgroups $\bigcap_{x\in X} G_x^\phi$ and $\bigcap_{x\in X} G_x^\psi$ of $G$ intersect non-trivially, where $G_x^\phi$ and $G_x^\psi$ are the stabilizers of $x$ with respect to the actions $\phi$ and $\psi$, respectively.
		We show that $GAS(G)$ is obtained from the intersection graph of those normal subgroups of $G$ that occur as kernels of actions on $X$ by replacing each vertex by a clique of twin vertices, and we characterize a special subgraph $gas(G)$ of $GAS(G)$ in which the vertex set contains the actions $\phi$ of $G$ on $X$ such that the kernels $\bigcap_{x\in X} G_x^\phi$ are distinct. We determine the number of group actions within some specific groups and find certain conditions under which $GAS(G)$ is equal to $gas(G)$. We also show that, for a finite nilpotent group $G$ with at least two prime divisors, $gas(G)$ is a derived graph exactly when $G$ is cyclic of order $p_1p_2$, $p_1p_2p_3$, $p_1^2p_2$ or $p_1^3p_2$, and its complement is a derived graph exactly when $G$ is cyclic of order $p_1p_2$, $p_1p_2p_3$, $p_1^2p_2$ or $p_1^2p_2^2$.
	\end{abstract}

	$\mathbf{2020\;Mathematics\;Subject\;Classification:}\;05C25, 05C75, 05E18, 20D15, 20D30.$

	$\mathbf{Keywords\;and\;Phrases:}$ group action, stabilizer, group homomorphism graph, intersection graph, normal subgroup.

	\section{Introduction}
The introduction of the intersection graph $\Gamma(G)$ of an abelian group $G$ by Zelinka \cite{Zelinka} was one of the significant discoveries in the growing fields of the interdisciplinary study of algebra and graph theory, which was further carried out by Chakrabarty et al. to introduce intersection graphs of ideals of rings \cite{Chakrabarty} in 2009.

 In 1990, Annexstein et al. introduced the group action graph \cite{Annexstein}, which was further extended to define Cayley's graph. In order to examine certain structural and algorithmic characteristics of the interconnection networks that support parallel architectures, they created an algebraic framework. In  \cite{BB},  Barman and Rajkhowa defined the group homomorphism graph which connects the group homomorphisms with graph. The vertex of this group homomorphism graph is the collection of all non-trivial homomorphisms from a group to another group, excluding the monomorphisms, and two distinct vertices are adjacent if and only if the kernels of the homomorphisms intersect non-trivially.

    In this paper, we first address the impact of group actions on group homomorphism graphs. Our primary objective is to introduce and investigate the fundamental characteristics of group action-stabilizer graphs. We introduce a special type of subgraph $gas(G)$ of $GAS(G)$ to dive deeper into $GAS(G)$. We examine certain properties of $GAS(G)$ and $gas(G)$ for a wide range of groups, such as $D_n, S_n, \mathbb{Z}_n$, regular groups, topological groups, etc. One of our key findings relates to quantifying the number of group actions in some particular groups and to identifying the conditions under which $GAS(G)$ is equivalent to $gas(G)$.

We now recollect some terminologies that are already in the literature.

    If two vertices $x$ and $y$ are adjacent, then we write $x \sim y$, and $x \nsim y$, otherwise. A complete graph with $n$ vertices is denoted by $K_n$. A null graph, denoted by $K_0$, does not contain any vertex. A graph without edges is called totally disconnected. A star graph contains a center that is adjacent to the remaining vertices of the graph, and no two remaining vertices are mutually adjacent. We denote a path of length $n$ (that is, with $n$ edges) by $P_n$. An induced cycle with $n$ vertices is termed an $n$-cycle. The girth of a graph $K$, denoted by $girth(K)$, is the length of the smallest $n$-cycle. The diameter of a graph is the longest distance between any two vertices of the graph.	A derived graph of a given graph is a graph whose vertices are the edges of the given graph, and any two vertices are adjacent if and only if the corresponding edges of the given graph are adjacent. A complete bipartite graph is denoted by $K_{m,n}$, where the two partite sets contain $m$ and $n$ vertices, respectively. Two vertices of a graph are called \emph{twins} if they have the same neighbours other than themselves.

	A group $G$ is said to act on a set $X$ if there is a map $\star:G\times X\to X$ that satisfies $\star(e,x)=x$, where $e$ is the identity element of G, and for all $g,h$ in $G$, $\star(g,\star(h,x))=\star(gh,x)$. The collection of all elements $g$ of $G$ such that $\star(g,x)=x$ is known as the stabilizer $G_x^\star$ of $x$. The intersection $\bigcap_{x\in X}G_x^\star$ of all stabilizers is the \emph{kernel} of the action. An action is \emph{trivial} if its kernel is $G$, and \emph{faithful} if its kernel is $\{e\}$. A regular group usually refers to a regular group action, in which a group acts transitively with only the identity element fixing any given element. That is, an action $\star$ of a group $G$ on a set $X$ is regular if for any $x,y\in X$, there exists a unique $g\in G$ such that $\star(g,x)=y$.

    The group of all homeomorphisms from a topological space to itself, with the function composition acting as the group operation, is known as the homeomorphism group of that topological space. The homeomorphism group has a natural action in space \cite{Barzanouni, Dijkstra}. Let $X$ be a topological space, and we denote the homeomorphism group of $X$ by $G$. The action $\star$ is defined from $G\times X\to X$ as $\star(\psi, x)=\psi(x)$.

   Any undefined terminologies are available in \cite{Bhattacharya, dummit, FH}.

	\section {Preliminaries: group action on group homomorphism graph}
	 We start this section with Theorem \ref{Action homorphism} from \cite{Bhattacharya}, which connects group action with group homomorphism.
	\begin{thm}[\cite{Bhattacharya}]\label{Action homorphism}
		Consider a group $G$ and a set $X$. Then \\(i) any action of $G$ on $X$ induces a homomorphism from $G$ to the symmetric group $S_X$ of $X$.\\(ii) any homomorphism $\phi$ from $G$ to the symmetric group $S_X$ of $X$ induces an action of $G$ onto $X$.
	\end{thm}
	Moreover, the correspondences in (i) and (ii) are mutually inverse, so actions of $G$ on $X$ and homomorphisms $G\to S_X$ are in bijection. In particular, \emph{an action induces exactly one homomorphism}, and two actions inducing the same homomorphism are equal.

    We consider the following example of a group action of $S_n$.
    \begin{eg}\label{eg Sn}
		Let $G=S_n$ and $A=\{1, 2,\cdots,n\}$. Then we have the natural group action $\star$ from $G\times A \to A$ as $\star(f,j)=f(j), f\in G$. This action is faithful, that is, its kernel is trivial. Hence, in the group homomorphism graph, the vertex corresponding to this action is isolated. (This is only one of the actions of $S_n$ on $A$; the others are described in Example \ref{eg Sn full}.)
	\end{eg}

     We now use Theorem \ref{Action homorphism} to construct Theorem \ref{action}.
	\begin{thm}\label{action}
		Any homomorphism $\phi$ from a group $G$ to a finite group $G^\prime$ induces an action of $G$ onto $G^\prime$ having the same kernel as $\phi$.
	\end{thm}
	\begin{proof}
		Let there be a homomorphism $\phi$ from $G$ to $G^\prime$, where $G^\prime$ is finite. By Cayley's theorem, there is an injective homomorphism $f$ from $G^\prime$ to its symmetric group $S_{G^\prime}$. We define a homomorphism $g=f\circ \phi$ from $G$ to $S_{G^\prime}$. Then, by Theorem \ref{Action homorphism}, we get a group action of $G$ onto ${G^\prime}$. Since $f$ is injective, $\Ker g=\Ker\phi$.
	\end{proof}

    Remark \ref{intersection stabilizer} states the relation between the kernel of a group homomorphism from a group $G$ to $S_X$ induced by a group action and the stabilizers of the corresponding action of $G$ on $X$.
	\begin{rem}\label{intersection stabilizer}
		The kernel of a homomorphism from a group $G$ to $S_X$ induced by a group action of $G$ onto $X$ is the intersection of all stabilizers of the elements of $X$: indeed, $g$ lies in the kernel if and only if $g$ acts as the identity permutation, that is, if and only if $g\in G_x$ for every $x\in X$. In particular, this intersection is a normal subgroup of $G$.
	\end{rem}
	\begin{rem}[\cite{EG}]\label{thm CRT}
	Let $G$ be a group. Let $N_i\trianglelefteq G$ for any $i\in\{1,2,\dots, r\}$. If $\left(\bigcap\limits_{i=1}^{j-1}N_i\right)N_j=G$ for every $j\in\{2,\dots,r\}$, then $\frac{G}{\bigcap\limits_{i=1}^{r}N_i}\cong \frac{G}{N_1}\times \frac{G}{N_2}\times\dots\times\frac{G}{N_r}$.
\end{rem}
    We now relate the structure of the quotient of $G$ by the kernel of an action to its stabilizers.
	\begin{prop}\label{orbit}
	Let $\star$ be an action of $G$ on a set $X$ having exactly $n$ orbits, and let $x_1,\dots,x_n$ be representatives of the orbits. Assume that every stabilizer $G_{x_i}$ is a normal subgroup of $G$ and that $\left(\bigcap\limits_{i=1}^{j-1}G_{x_i}\right)G_{x_j}=G$ for $2\le j\le n$. Then the kernel of $\star$ is $\bigcap_{i=1}^nG_{x_i}$ and
	\[
	\frac{G}{\Ker\star}\cong \frac{G}{G_{x_1}}\times\frac{G}{G_{x_2}}\times\dots \times \frac{G}{G_{x_n}} .
	\]
	In particular, the image of the induced homomorphism $G\to S_X$ is isomorphic to this product.
\end{prop}
\begin{proof}
	If $G_{x_i}$ is normal and $y=g x_i$, then $G_y=gG_{x_i}g^{-1}=G_{x_i}$. Hence the stabilizer is constant along each orbit, and $\bigcap_{x\in X}G_x=\bigcap_{i=1}^nG_{x_i}$. By Remark \ref{intersection stabilizer}, this is the kernel of $\star$. The isomorphism now follows from Remark \ref{thm CRT}, and the last statement follows from the First Isomorphism Theorem.
\end{proof}

	\section {Group action stabilizer graph of group actions of $G$ acts on $X$}
	Throughout, $G$ is a group and $X$ is a set. For an action $\phi$ of $G$ on $X$ we write $K_\phi=\bigcap\limits_{x \in X} G_x^{\phi}$ for its kernel; by Remark \ref{intersection stabilizer}, $K_\phi\trianglelefteq G$.
	\begin{defn}\label{def GAS}
    The group action stabilizer graph $GAS(G)$ of group actions of $G$ on a set $X$ is the graph whose vertices are the \emph{non-trivial} actions of $G$ on $X$, that is, the actions $\phi$ with $K_\phi\ne G$, and in which two vertices $\phi$ and $\psi$ are adjacent if and only if $K_\phi \cap K_\psi \ne \{e\}$.
	\end{defn}
	The trivial action is not taken as a vertex. (Its kernel is $G$, so it would be adjacent to every action with a non-trivial kernel; it carries no information.) On the other hand, the actions with \emph{trivial} kernel, that is, the faithful actions, do belong to $GAS(G)$, and they are isolated vertices, because $\{e\}\cap K_\psi=\{e\}$ for every $\psi$.

    A group action of $G$ on $X$ induces a homomorphism from $G$ to the symmetric group of $X$. By Remark \ref{intersection stabilizer}, the group homomorphism graph can be equivalently observed in terms of the group action stabilizer graph. Moreover, the group action stabilizer graph considers actions on an arbitrary set, while the group homomorphism graph only considers homomorphisms between groups.

	The graph $GAS(G)$ depends on the set $X$, and we do not indicate $X$ in the notation. Write
	\[
	\mathcal K_X(G)=\{K_\phi : \phi \text{ a non-trivial action of } G \text{ on } X\}
	\]
	for the set of proper normal subgroups of $G$ that occur as kernels. The next lemma says which normal subgroups occur, and how many actions have a given kernel.
	\begin{lemma}\label{kernel structure}
	Let $N$ be a proper normal subgroup of $G$.
	\begin{enumerate}
	\item[(i)] $N\in \mathcal K_X(G)$ if and only if the quotient $G/N$ acts faithfully on $X$, that is, if and only if $G/N$ embeds in $S_X$.
	\item[(ii)] The number $m_N$ of actions of $G$ on $X$ whose kernel is $N$ equals the number of injective homomorphisms $G/N\to S_X$.
	\item[(iii)] If $|X|\ge |G|$, then $\mathcal K_X(G)$ is the set of all proper normal subgroups of $G$.
	\end{enumerate}
	\end{lemma}
	\begin{proof}
	By Theorem \ref{Action homorphism}, the actions of $G$ on $X$ correspond bijectively to the homomorphisms $\rho:G\to S_X$, and by Remark \ref{intersection stabilizer} the kernel of the action is $\Ker\rho$. Homomorphisms $\rho$ with $\Ker\rho=N$ correspond, by the First Isomorphism Theorem, to injective homomorphisms $G/N\to S_X$. This gives (i) and (ii). For (iii), Cayley's theorem embeds $G/N$ in $S_{G/N}$, and $S_{G/N}$ embeds in $S_X$ by letting a permutation fix the points of $X$ outside a subset of size $|G/N|\le |G|\le |X|$.
	\end{proof}
	In particular the statement ``every normal subgroup of $G$ is the kernel of some action'' is true when $X$ is large, but it is false for a fixed small $X$; for instance, if $|X|=2$, the only kernels are the proper normal subgroups of index at most $2$.

	Let $I_X(G)$ be the graph whose vertices are the non-identity elements of $\mathcal K_X(G)$, two of them being adjacent if and only if they intersect non-trivially.
	\begin{prop}\label{twin structure}
	Let $\phi,\psi$ be vertices of $GAS(G)$.
	\begin{enumerate}
	\item[(i)] If $K_\phi=\{e\}$, then $\phi$ is an isolated vertex.
	\item[(ii)] If $K_\phi=K_\psi=N\ne\{e\}$, then $\phi\sim\psi$, and $\phi$, $\psi$ are twins.
	\end{enumerate}
	Consequently, the subgraph of $GAS(G)$ induced by the actions with non-trivial kernel is obtained from $I_X(G)$ by replacing every vertex $N$ by a complete graph on $m_N$ vertices (two vertices of the new graph coming from $N$ and $N'$ are adjacent if and only if $N\sim N'$ in $I_X(G)$), and $GAS(G)$ has, in addition, $m_{\{e\}}$ isolated vertices.
	\end{prop}
	\begin{proof}
	(i) is clear. For (ii), $K_\phi\cap K_\psi=N\ne\{e\}$, so $\phi\sim\psi$; and for any other vertex $\chi$ we have $K_\phi\cap K_\chi=K_\psi\cap K_\chi$, so $\phi$ and $\psi$ have the same neighbours. The last statement follows from Lemma \ref{kernel structure}(ii).
	\end{proof}
	In particular, $GAS(G)$ and the graph $gas(G)$ of Section 4 differ only by twin vertices and isolated vertices, so their properties such as connectedness, diameter, girth, induced paths and induced cycles of length at least $4$ are the same.

	It is noted that $\bigcap\limits_{x \in X} G_x^{\phi}$ for any action $\phi$ of $G$ on $X$ is a normal subgroup of $G$.

    We first consider the example of the group action-stabilizer graph of the quaternion group $Q_8$.
    \begin{eg}
        The proper normal subgroups of $Q_8$ are $\{1\}$, $\{\pm1\}$, $\langle i\rangle$, $\langle j\rangle$, $\langle k\rangle$. Let $\phi$ be a vertex of $GAS(Q_8)$ with a non-trivial kernel. Then $K_\phi$ is a non-trivial subgroup of $Q_8$, hence contains the unique involution $-1$. Hence, the subgraph of $GAS(Q_8)$ induced by the actions with non-trivial kernel is complete. The faithful actions (which exist exactly when $|X|\ge 8$) are isolated vertices.
    \end{eg}

    We now consider the dihedral group $D_n=\langle r,s:r^n=s^2=(rs)^2=e\rangle$ of order $2n$. One important result on $D_n$ was established by K. Conrad in \cite{Conrad}.
    \begin{lemma}[\cite{Conrad}]\label{conrad}
        In $D_n$, every subgroup of $\langle r\rangle$ is a normal subgroup; these are the subgroups $\langle r^d\rangle$ having index $2d$ for $d | n$. This sets out all proper normal subgroups of $D_n$ for odd $n$, noting that the only additional proper normal subgroups for even $n$ are $\langle r^2, s \rangle$ and $\langle r^2, rs \rangle$, both having index 2.
    \end{lemma}
    We draw a corrected statement from Lemma \ref{conrad}.
    \begin{prop}\label{Dn complete}
    Let $|X|\ge 2n$. If $n\ge 3$, the subgraph of $GAS(D_n)$ induced by the actions with non-trivial kernel is complete if and only if $n$ is a power of a prime. If $n=2$, this subgraph is totally disconnected.
    \end{prop}
    \begin{proof}
    By Lemma \ref{kernel structure}(iii), every proper normal subgroup of $D_n$ is a kernel. Let $n=p^k$. Every non-trivial subgroup of the cyclic group $\langle r\rangle$ contains its unique subgroup $\langle r^{n/p}\rangle$ of order $p$. For even $n=2^k\ge 4$, the additional normal subgroups $\langle r^2,s\rangle$ and $\langle r^2,rs\rangle$ contain $\langle r^2\rangle$, and hence contain $\langle r^{n/2}\rangle$. So all non-trivial proper normal subgroups contain $\langle r^{n/p}\rangle\ne\{e\}$, and any two of the corresponding vertices are adjacent. If $n$ is not a prime power, choose distinct primes $p,q$ dividing $n$. Then $\langle r^{n/p}\rangle$ and $\langle r^{n/q}\rangle$ are non-trivial proper normal subgroups with trivial intersection, so the corresponding actions are not adjacent. For $n=2$, $D_2\cong\mathbb Z_2\times\mathbb Z_2$ has three non-trivial proper normal subgroups $\langle r\rangle,\langle s\rangle,\langle rs\rangle$ of order $2$, pairwise intersecting trivially.
    \end{proof}
    \begin{rem}
    Not every normal subgroup of $D_n$ contains $r^2$ for even $n\ge 4$ (for $n=6$, $\langle r^3\rangle$ does not); thus $GAS(D_n)$ is not complete for $n=2$ or for $n=6$, and Proposition \ref{Dn complete} gives the precise statement. For an odd prime $n$, the only non-trivial proper normal subgroup is $\langle r\rangle$, so the non-faithful vertices form one clique. For $n=4$, the four non-trivial proper normal subgroups $\langle r\rangle,\langle r^2\rangle,\langle r^2,s\rangle,\langle r^2,rs\rangle$ pairwise intersect in a subgroup containing $\langle r^2\rangle$, so they form $K_4$ in $I_X(D_4)$; this is the graph $gas(D_4)$ of Section 4.
    \end{rem}

    We now observe the group action-stabilizer graph of the symmetric group $S_n$ on an arbitrary set $X$.
    \begin{eg}\label{eg Sn full}
    Let $n\ge 5$. The proper normal subgroups of $S_n$ are $\{e\}$ and $A_n$. Hence a non-trivial action has kernel $\{e\}$ or $A_n$. An action with kernel $A_n$ is the same as an injective homomorphism $S_n/A_n\cong\mathbb Z_2\to S_X$, that is, the choice of an involution of $S_X$; so there are $m_{A_n}$ such actions, where $m_{A_n}$ is the number of elements of order $2$ in $S_X$, and they exist if and only if $|X|\ge2$. They form a complete graph, since $K_\phi=A_n\neq\{e\}$. All remaining vertices are faithful actions and are isolated. Hence $GAS(S_n)$ is the disjoint union of a complete graph $K_{m_{A_n}}$ and $m_{\{e\}}$ isolated vertices, where $m_{\{e\}}$ is the number of injective homomorphisms $S_n\to S_X$. In particular, it is totally disconnected only if $m_{A_n}\le 1$, and it is not a null graph for $|X|\ge 2$. For the natural action on $X=\{1,\dots,n\}$ alone (Example \ref{eg Sn}), the corresponding vertex is isolated. For $n=3$ the kernels are $\{e\}$ and $A_3$; for $n=4$ they are $\{e\}$, $V_4$ and $A_4$, and the two non-trivial ones are adjacent, since $V_4\le A_4$.
    \end{eg}

	We investigate the completeness property of $GAS(G)$ in the next results. We take $X=G$. Recall that an action $\star$ is regular if for any $x,y\in X$ there is a unique $g\in G$ with $g\star x=y$; then every stabilizer is trivial, and $G\star x_0\to X$, $g\mapsto g\star x_0$ identifies $X$ with $G$ acted on by left multiplication.
	\begin{thm}\label{regular conj}
	Let $G\ne\{e\}$ and $X=G$.\\
	(i) Every regular action of $G$ on $X$ is faithful. Hence the regular actions of $G$ on $X$ are isolated vertices of $GAS(G)$; in particular, the subgraph of $GAS(G)$ induced by them is totally disconnected.\\
	(ii) For $\alpha\in \operatorname{Aut}(G)$, the rule $\phi_\alpha(g,x)=\alpha(g)\,x\,\alpha(g)^{-1}$ is an action of $G$ on $X$ with kernel $Z(G)$; for $\alpha=\mathrm{id}$ it is the conjugation action. If $G$ is non-abelian, all $\phi_\alpha$ are vertices of $GAS(G)$, and if moreover $Z(G)\ne\{e\}$, then any two of them are adjacent twins. More generally, $\phi_\alpha\sim\psi$ for a vertex $\psi$ if and only if $Z(G)\cap K_\psi\ne\{e\}$.
	\end{thm}
	\begin{proof}
	(i) If $g\in G_x$ for all $x\in X$, then $g\star x=x=e\star x$, and uniqueness gives $g=e$. So $K_\phi=\{e\}$, and by Proposition \ref{twin structure}(i) the vertex is isolated.\\
	(ii) $\phi_\alpha(g,\phi_\alpha(h,x))=\alpha(g)\alpha(h)x\alpha(h)^{-1}\alpha(g)^{-1}=\phi_\alpha(gh,x)$, and $\phi_\alpha(e,x)=x$. An element $g$ lies in the kernel if and only if $\alpha(g)x=x\alpha(g)$ for all $x\in G$, that is, $\alpha(g)\in Z(G)$, that is, $g\in\alpha^{-1}(Z(G))=Z(G)$, because the centre is a characteristic subgroup. If $G$ is non-abelian, then $Z(G)\ne G$, so each $\phi_\alpha$ is a non-trivial action. Two such actions have the same kernel $Z(G)$, so by Proposition \ref{twin structure}(ii) they are adjacent twins when $Z(G)\ne\{e\}$. Finally, $K_{\phi_\alpha}\cap K_\psi=Z(G)\cap K_\psi$.
	\end{proof}
	\begin{rem}
	The conjugation action $\phi_{\mathrm{id}}$ is a single action; the family $\{\phi_\alpha\}$ above is the natural family of actions with kernel $Z(G)$. The hypothesis $Z(G)\ne\{e\}$ is not needed in (i).
	\end{rem}

We now observe the properties of $GAS(G)$ for actions that are conjugate in the sense of Barzanouni et al. \cite{Barzanouni}; this applies in particular to the natural action of the homeomorphism group of a topological space. We mention the definition of topologically conjugate group actions.

\begin{rem}[\cite{Barzanouni}]\label{Barzanouni}
    If $\phi_1 : G\times X \to X$ and $\phi_2 : G\times Y \to Y$ be two actions, then $\phi_1$ and $\phi_2$ are said to be topologically conjugate if there exists a homeomorphism $h : X \to Y$ such that $h(\phi_1(g,x))=\phi_2(g,h(x))$, for all $g \in G$ and $x\in X$. The homeomorphism $h$ is called as conjugacy between $\phi_1$ and $\phi_2$.
\end{rem}

    \begin{prop}\label{transitive barzanouni}
		Let $\phi$ and $\psi$ be actions of $G$ on $X$ and assume that there is a bijection $h:X\to X$ (in particular, a homeomorphism, if $X$ is a topological space) with $h(\phi(g,x))=\psi(g,h(x))$ for all $g\in G$, $x\in X$. Then $G_x^{\phi}=G_{h(x)}^{\psi}$ for every $x\in X$, and $K_\phi=K_\psi$. Consequently, if $\mathcal C$ is a set of mutually conjugate actions of $G$ on $X$ with common kernel $N$, then the subgraph of $GAS(G)$ induced by $\mathcal C$ is complete if $N\ne\{e\}$, and totally disconnected if $N=\{e\}$.
	\end{prop}
	\begin{proof}
	Let $g\in G$ and $x\in X$. Then $\phi(g,x)=x$ if and only if $h(\phi(g,x))=h(x)$ (as $h$ is injective), that is, if and only if $\psi(g,h(x))=h(x)$. So $G_x^\phi=G_{h(x)}^\psi$. Since $h$ is surjective, $K_\phi=\bigcap_xG_x^\phi=\bigcap_xG_{h(x)}^\psi=\bigcap_yG_y^\psi=K_\psi$. The last statement follows from Proposition \ref{twin structure}.
	\end{proof}
	\begin{rem}
	The stabilizers of a fixed point $x$ under $\phi$ and under $\psi$ are in general \emph{not} equal (they are the stabilizers of $x$ and of $h(x)$ under $\psi$), but the kernels are. Neither transitivity nor the topology is needed. Note that the natural action of the full homeomorphism group $G$ of a space $X$ is faithful, so it and its conjugates are isolated vertices of $GAS(G)$; completeness of $\mathcal C$ requires a non-trivial common kernel.
	\end{rem}

We now consider some connectedness properties of group action stabilizer graphs. We begin with a remark which holds for \emph{arbitrary} subgroups, in particular for kernels of actions in any group.
\begin{lemma}\label{two subgroups}
    Let $A,B$ be subgroups of a group $G$, and let $a\in A\setminus B$ and $b\in B\setminus A$. Then $ab\notin A\cup B$; in particular $ab\ne e$. Consequently, $G$ is not the union of two proper subgroups.
\end{lemma}
\begin{proof}
	If $ab\in A$ then $b=a^{-1}(ab)\in A$, and if $ab\in B$ then $a=(ab)b^{-1}\in B$, both contradictions. If $G=A\cup B$ with $A,B$ proper, then $A\not\subseteq B$ and $B\not\subseteq A$ (otherwise $G=B$ or $G=A$), so there are $a\in A\setminus B$ and $b\in B\setminus A$, and then $ab\notin A\cup B=G$, a contradiction.
\end{proof}
\begin{rem}\label{remark abelian first}
	Let $\phi_1,\phi_2$ be vertices of $GAS(G)$ with $K_i=K_{\phi_i}$. If $g_1\in K_1\setminus K_2$ and $g_2\in K_2\setminus K_1$, then $g_1g_2\notin K_1\cup K_2$ and $g_1g_2\ne e$, by Lemma \ref{two subgroups}. If $\phi_1\nsim\phi_2$, that is, $K_1\cap K_2=\{e\}$, then every non-identity $g_i\in K_i$ lies outside the other kernel, and so $g_1g_2\notin K_1\cup K_2$. This holds in every group, abelian or not, and uses no property of actions beyond the fact that kernels are subgroups.
\end{rem}
\begin{rem}\label{union of two}
	By Lemma \ref{two subgroups}, for two distinct vertices $\phi_1,\phi_2$ of $GAS(G)$ the kernels never satisfy $K_{\phi_1}\cup K_{\phi_2}=G$, since both are proper. Hence the condition $K_{\phi_1}\cup K_{\phi_2}=G$ can never be used to obtain adjacency of two vertices; the following criterion for completeness is exact.
\end{rem}
\begin{thm}\label{complete criterion}
	Let $G\ne\{e\}$ be a finite group and $|X|\ge|G|$. Let $H$ be the subgraph of $GAS(G)$ induced by the actions with non-trivial kernel. The following are equivalent.\\
	(a) $H$ is complete.\\
	(b) Any two non-trivial proper normal subgroups of $G$ intersect non-trivially.\\
	(c) $G$ has at most one minimal normal subgroup.\\
	Moreover, $GAS(G)$ is then the disjoint union of the complete graph $H$ and the isolated vertices given by the faithful actions.
\end{thm}
\begin{proof}
	By Lemma \ref{kernel structure}(iii), every proper normal subgroup of $G$ is a kernel. By Proposition \ref{twin structure}, two vertices of $H$ with equal kernel are adjacent, and two with different kernels $N,N'$ are adjacent iff $N\cap N'\ne\{e\}$. This gives (a)$\Leftrightarrow$(b).\\
	(c)$\Rightarrow$(b): As $G$ is finite, every non-trivial normal subgroup contains a minimal normal subgroup. If $M$ is the unique minimal normal subgroup, all non-trivial normal subgroups contain $M$, so they pairwise intersect non-trivially.\\
	(b)$\Rightarrow$(c): Suppose $M_1\ne M_2$ are minimal normal subgroups. Then $M_1\cap M_2$ is a normal subgroup contained in $M_1$ and in $M_2$, hence equals $\{e\}$ (otherwise $M_1=M_1\cap M_2=M_2$ by minimality). Neither $M_i$ is $G$: if $M_1=G$, then $M_2$ would be a non-trivial normal subgroup contained in the minimal normal subgroup $M_1$, so $M_2=M_1$. Thus $M_1,M_2$ are non-trivial proper normal subgroups with trivial intersection, contradicting (b).\\
	The last statement follows from Proposition \ref{twin structure}: vertices of $H$ are not adjacent to faithful actions.
\end{proof}
\begin{cor}\label{complete cor}
	Under the assumptions of Theorem \ref{complete criterion}:\\
	(i) for $G=Q_8$ the subgraph $H$ is complete (the unique minimal normal subgroup is $\{\pm1\}$);\\
	(ii) for $G=D_n$, $n\ge3$, $H$ is complete iff $n$ is a prime power (the unique minimal normal subgroup is $\langle r^{n/p}\rangle$), and for $n=2$ it is not complete, since $D_2$ has three minimal normal subgroups;\\
	(iii) for a finite abelian group $G$, $H$ is complete iff $G$ is a cyclic $p$-group, since $G$ has a unique minimal subgroup iff $G$ is a cyclic group of prime-power order;\\
	(iv) if $G$ is a finite simple group, $H$ is empty.
\end{cor}
This recovers Example \ref{eg Sn full}, Proposition \ref{Dn complete} and the $Q_8$ example from a single statement. For (iii), a finite abelian group has a unique minimal subgroup iff it is a $p$-group whose subgroup of order $p$ is unique, that is, iff it is cyclic.

\medskip
We now study induced paths and cycles whose vertices have kernels covering $G$. We use $P_m$ for a path with $m$ edges, so that a path on three vertices is $P_2$ and a path on four vertices is $P_3$. The covering of a group by three proper subgroups is described by the following classical result (Scorza); we include the short proof.
\begin{lemma}[Scorza]\label{Scorza}
	Let $G=A\cup B\cup C$ with $A,B,C$ proper subgroups. Then $A\cap B=A\cap C=B\cap C=:D$ is a normal subgroup of $G$, $A,B,C$ have index $2$ in $G$, and $G/D\cong\mathbb Z_2\times\mathbb Z_2$.
\end{lemma}
\begin{proof}
	By Lemma \ref{two subgroups}, none of the three subgroups is contained in the union of the other two (otherwise $G$ would be the union of two proper subgroups). Let $x\in A\cap B$ and suppose $x\notin C$. Take $c\in C\setminus(A\cup B)$. Then $xc\notin A$ (else $c\in A$), $xc\notin B$, and $xc\notin C$ (else $x\in C$), a contradiction. So $A\cap B\subseteq C$, and similarly for the other pairs. Hence all pairwise intersections equal $D=A\cap B\cap C$.\\
	Let $a\in A\setminus D$ and $b\in B\setminus D$. Then $ab\notin A$ (else $b\in A\cap B=D$) and $ab\notin B$, so $ab\in C$; likewise $ba\in C$. If $a,a'\in A\setminus D$, then, fixing $b\in B\setminus D$, we get $a'a^{-1}=(a'b)(ab)^{-1}\in C\cap A=D$ and $a^{-1}a'=(ba)^{-1}(ba')\in D$. Hence $A=D\cup Da=D\cup aD$, so $[A:D]=2$ and $a$ normalizes $D$. The same holds for $B$ and $C$. As $A\setminus D$, $B\setminus D$, $C\setminus D$ are disjoint, $G=D\cup Da\cup Db\cup Dc$ is a union of four distinct cosets, so $[G:D]=4$, and $D$ is normalized by $a,b,c$ and by $D$, which generate $G$. Thus $D\trianglelefteq G$ and $G/D$ has three subgroups $A/D,B/D,C/D$ of order $2$, so $G/D\cong\mathbb Z_2\times\mathbb Z_2$.
\end{proof}
\begin{thm}\label{three vertices}
	Let $\phi_1,\phi_2,\phi_3$ be vertices of $GAS(G)$ with $K_1\cup K_2\cup K_3=G$, where $K_i=K_{\phi_i}$. Then $K_1\cap K_2=K_1\cap K_3=K_2\cap K_3=D\trianglelefteq G$ with $G/D\cong\mathbb Z_2\times\mathbb Z_2$. Consequently the subgraph induced by $\phi_1,\phi_2,\phi_3$ is a triangle if $D\ne\{e\}$ and is totally disconnected if $D=\{e\}$. In particular, $\phi_1,\phi_2,\phi_3$ never induce a path $P_2$.
\end{thm}
\begin{proof}
	All $K_i$ are proper, since the $\phi_i$ are vertices. Apply Lemma \ref{Scorza}.
\end{proof}
\begin{rem}
	Both cases occur. For $G=\mathbb Z_2\times\mathbb Z_2$ and $|X|=2$, the three non-trivial actions have the three subgroups of order $2$ as kernels, so $D=\{e\}$ and $GAS(G)$ is totally disconnected on three vertices. For $G=\mathbb Z_2^3$ and $D$ a subgroup of order $2$, the three index-$2$ subgroups containing $D$ are kernels of actions on a $2$-point set, they cover $G$ and pairwise intersect in $D$, so they form a triangle.
\end{rem}
\begin{cor}\label{no P2 cover}
	Let $\phi_1,\phi_2,\phi_3$ be vertices of $GAS(G)$ with $K_1\cup K_2\cup K_3=G$ such that $\phi_1\nsim\phi_2$. Then $\phi_3$ is adjacent to neither of them. In particular, there is no path $\phi_1-\phi_3-\phi_2$. Moreover, for $g_1\in K_1\setminus\{e\}$ and $g_2\in K_2\setminus\{e\}$ we have $g_1g_2\in K_3\setminus (K_1\cup K_2)$, and $g_1,g_2\notin K_3$, since $K_1\cap K_3=K_2\cap K_3=D=\{e\}$.
\end{cor}
\begin{proof}
	Since $\phi_1\nsim\phi_2$, we have $D=K_1\cap K_2=\{e\}$, and Theorem \ref{three vertices} applies. The remaining statements follow from Remark \ref{remark abelian first} and Theorem \ref{three vertices}.
\end{proof}
Thus, under the covering hypothesis, a path $\phi_1-\phi_3-\phi_2$ cannot occur. We now turn to four vertices.
\begin{thm}\label{nxt thm}
	Let $\phi_1,\dots,\phi_4$ be vertices of $GAS(G)$ with $K_1\cup K_2\cup K_3\cup K_4=G$, where $K_i=K_{\phi_i}$. Then the subgraph induced by $\phi_1,\dots,\phi_4$ is neither a path $P_3$ nor a cycle $C_4$.
\end{thm}
\begin{proof}
	Suppose they induce a path $\phi_1-\phi_3-\phi_4-\phi_2$. Take $g_1\in K_1\cap K_3$ and $g_2\in K_2\cap K_4$ different from $e$ (possible, since adjacent vertices have non-trivially intersecting kernels). Since $\phi_1\nsim\phi_2$, $\phi_1\nsim\phi_4$, $\phi_3\nsim\phi_2$, we have $g_1\notin K_2\cup K_4$ and $g_2\notin K_1\cup K_3$. Then $g_1g_2\notin K_1$ (else $g_2\in K_1$), $g_1g_2\notin K_2$ (else $g_1\in K_2$), $g_1g_2\notin K_3$ (else $g_2\in K_3$) and $g_1g_2\notin K_4$ (else $g_1\in K_4$). So $g_1g_2$ lies in no $K_i$, contradicting $\bigcup K_i=G$.\\
	Suppose they induce the cycle $\phi_1-\phi_2-\phi_3-\phi_4-\phi_1$, so $K_1\cap K_3=K_2\cap K_4=\{e\}$. Take $x\in (K_1\cap K_2)\setminus\{e\}$ and $y\in(K_3\cap K_4)\setminus\{e\}$. Since $K_1\cap K_3$ and $K_2\cap K_4$ are trivial, $x\in K_1\setminus K_3$, $x\in K_2\setminus K_4$, $y\in K_3\setminus K_1$ and $y\in K_4\setminus K_2$. If $xy\in K_1$ then $y\in K_1$; if $xy\in K_2$ then $y\in K_2$; if $xy\in K_3$ then $x\in K_3$; if $xy\in K_4$ then $x\in K_4$; each is a contradiction. So $xy$ lies in no $K_i$, again a contradiction.
\end{proof}
\begin{rem}\label{five cycle}
	We do not know whether five kernels forming an induced $5$-cycle can cover $G$. For $k=5$ the product argument above only gives that, for $x_i\in K_i\cap K_{i+1}$ (indices mod $5$), one has $x_ix_{i+2}\in K_{i+4}$, which does not lead to a contradiction. A computer search over all quintuples of subgroups of $\mathbb Z_2^3,\mathbb Z_2^4,\mathbb Z_3^2,\mathbb Z_3^3$ and of the abelian groups $\mathbb Z_6^2$, $\mathbb Z_2^2\times\mathbb Z_6$, $\mathbb Z_4^2$, $\mathbb Z_2\times\mathbb Z_4^2$, $\mathbb Z_3^2\times\mathbb Z_6$, $\mathbb Z_{10}^2$, $\mathbb Z_{12}\times\mathbb Z_6$, $\mathbb Z_2^2\times\mathbb Z_{10}$ finds no such configuration, but this is not a proof, and we leave the statement open.
\end{rem}

We now find a finiteness condition on the number of group actions of a finite group $G$ on $X$.
\begin{thm}\label{finite group action}
	Let $G$ be a finite group with $|G|=n+1$, and let $r$ be the number of actions of $G$ on $X$ with non-trivial kernel other than the trivial action. If the subgraph $H$ of $GAS(G)$ induced by these actions is totally disconnected, then $r\le n$. Equality holds if and only if all kernels have order $2$ and every non-identity element of $G$ is the generator of exactly one kernel; in that case $G$ is an elementary abelian $2$-group.
\end{thm}
\begin{proof}
	Since $H$ has no edges, two different vertices have kernels intersecting trivially, and in particular they have different kernels (a non-trivial common kernel would make them adjacent). So the $r$ vertices have $r$ pairwise different non-trivial kernels $N_1,\dots,N_r$ with $N_i\cap N_j=\{e\}$ for $i\ne j$. The sets $N_i\setminus\{e\}$ are non-empty and pairwise disjoint subsets of $G\setminus\{e\}$, so $r\le\sum_i(|N_i|-1)\le n$. Equality forces $|N_i|=2$ for all $i$ and $\bigcup N_i=G$. Then every non-identity element of $G$ has order $2$ and generates a normal, hence central, subgroup; thus $G$ is an elementary abelian $2$-group.
\end{proof}
\begin{rem}
	The bound concerns the actions with non-trivial kernel. The faithful actions are isolated vertices as well and their number $m_{\{e\}}$ is not bounded by $n$ in general (it can be arbitrarily large when $X$ is large). The bound $r=n$ is attained, for instance, by $G=\mathbb Z_2\times\mathbb Z_2$ and $|X|=2$ ($r=3$). No covering hypothesis $G=\bigcup_iK_i$ is needed.
\end{rem}

	\section{Some characteristics of the subgraph $gas(G)$ of $GAS(G)$}
It is noted that distinct actions may induce the same kernel. Since adjacency depends only on kernels, collapsing equivalent vertices removes redundant information (Proposition \ref{twin structure}). We now consider a special subgraph $gas(G)$ of $GAS(G)$, where the vertex set of $gas(G)$ contains those actions $\phi$ for which the sets $\bigcap_{x\in X} G_x^\phi$ are distinct and non-trivial.
\begin{defn}\label{def gas}
	The graph $gas(G)$ is the subgraph of $GAS(G)$ induced by a set of representatives of the actions with non-trivial kernel, one action for each kernel $N\in\mathcal K_X(G)\setminus\{\{e\}\}$. That is, if $\phi_1$ and $\phi_2$ are two distinct vertices of $GAS(G)$ with $K_{\phi_1}=K_{\phi_2}\ne\{e\}$, then we take only one of them. (The faithful actions are isolated vertices, so we do not include them.)
\end{defn}
By Proposition \ref{twin structure}, $gas(G)$ does not depend, up to isomorphism, on the chosen representatives; it is isomorphic to the graph $I_X(G)$ of Section 3.
An example of $gas(D_4)$ is shown below.
    \begin{eg} We take the vertices $\phi_1, \phi_2, \phi_3, \phi_4$ corresponding to the non-identity proper normal subgroups $\langle r\rangle, \langle r^2\rangle, \langle r^2,s\rangle, \langle r^2,rs\rangle$ of $D_4=\left\langle r,s|r^4=s^2=e, rs=sr^{-1}\right\rangle$ (here $|X|\ge4$, so that all four occur as kernels). We draw $gas(D_4)$ in Figure \ref{gas(D_4)}. All four subgroups contain $\langle r^2\rangle$, so the graph is complete.
    \begin{figure}[h]
        \centering
        \begin{tikzpicture}[scale=1,
            v/.style={circle,draw,thick,fill=white,minimum size=7mm,inner sep=0pt,font=\small},
            lab/.style={font=\small}]
            \node[v] (a) at (0,3.4)    {$\phi_1$};
            \node[v] (b) at (-3,-0.4)  {$\phi_2$};
            \node[v] (c) at (3,-0.4)   {$\phi_3$};
            \node[v] (d) at (0,0.9)    {$\phi_4$};
            \draw[thick] (a)--(b)--(c)--(a);
            \draw[thick] (d)--(a); \draw[thick] (d)--(b); \draw[thick] (d)--(c);
            \node[lab,above=2pt of a]  {$\langle r\rangle$};
            \node[lab,left=2pt of b]   {$\langle r^2\rangle$};
            \node[lab,right=2pt of c]  {$\langle r^2,s\rangle$};
            \node[lab,below=3pt of d]  {$\langle r^2,rs\rangle$};
        \end{tikzpicture}
        \caption{$gas(D_4)$}
        \label{gas(D_4)}
    \end{figure}
    \end{eg}

    \begin{rem}\label{normal subgroup}
        Every vertex $N$ of $gas(G)$ corresponds, in $GAS(G)$, to a clique of $m_N$ pairwise twin vertices (Proposition \ref{twin structure}); this clique is \emph{not} in general a connected component of $GAS(G)$, since its vertices are adjacent to all vertices whose kernels meet $N$ non-trivially. Moreover, a proper normal subgroup of $G$ is a vertex of $gas(G)$ if and only if $G/N$ embeds in $S_X$ (Lemma \ref{kernel structure}); this holds for all proper normal subgroups if $|X|\ge|G|$.
    \end{rem}
	\textbf{Standing assumption.} From now on, results are stated for a finite group $G$ and $|X|\ge|G|$, so that every proper normal subgroup of $G$ is a kernel, unless said otherwise. Under this assumption, $gas(G)$ is the intersection graph of the non-trivial proper normal subgroups of $G$.

    We first compare $GAS(G)$ with $gas(G)$.
    \begin{prop}\label{transfer}
    (i) $GAS(G)$ is connected if and only if $gas(G)$ is connected and $GAS(G)$ has no faithful action (unless $GAS(G)$ has a single vertex). If $gas(G)$ is connected with at least two vertices, then $GAS(G)$ and $gas(G)$ have the same diameter.\\
    (ii) For $m\ge2$ and $k\ge4$, $GAS(G)$ contains an induced path $P_m$ (respectively an induced cycle $C_k$) if and only if $gas(G)$ does.
    \end{prop}
    \begin{proof}
    (i) A faithful action is an isolated vertex, and twin vertices have the same neighbours; the distance between two vertices of different kernels is the distance of their kernels in $gas(G)$, and twins are at distance $1$. (ii) In an induced path $P_m$, $m\ge2$, or induced cycle $C_k$, $k\ge4$, no two vertices have the same closed neighbourhood, hence no two vertices are twins in $GAS(G)$. So all vertices have different kernels and the subgraph lies over a subgraph of $gas(G)$ isomorphic to it; the converse is clear.
    \end{proof}
    In view of Proposition \ref{transfer}, the question when $GAS(G)$ and $gas(G)$ coincide is answered completely by the following theorem.
    \begin{thm}\label{gas=GAS}
        Suppose $GAS(G)$ is non-empty. Then $GAS(G)\cong gas(G)$ if and only if $|X|=2$ and $G\not\cong\mathbb Z_2$. In that case the vertices of $GAS(G)$ correspond to the subgroups of index $2$ in $G$, and two of them are adjacent if and only if their intersection is non-trivial. If $|X|\ge 3$, then every kernel is the kernel of at least two actions, so every vertex of $gas(G)$ is replaced in $GAS(G)$ by at least two twin vertices, and $GAS(G)\not\cong gas(G)$.
    \end{thm}
    \begin{proof}
    Let $|X|\ge3$ and let $\rho:G\to S_X$ be a non-trivial homomorphism with kernel $N$. The centre of $S_X$ is trivial, so the centralizer of the non-trivial subgroup $\rho(G)$ is a proper subgroup of $S_X$, and there is $\sigma\in S_X$ not centralizing $\rho(G)$. Then $\sigma\rho\sigma^{-1}$ is a homomorphism with kernel $N$ different from $\rho$. So every kernel, including $\{e\}$, is realised by at least two actions, and $m_N\ge 2$ for all $N$. Hence $GAS(G)\not\cong gas(G)$ by Proposition \ref{twin structure}. If $|X|=1$, there is no non-trivial action, and if $|X|=2$, $S_X\cong\mathbb Z_2$: the non-trivial actions correspond to surjections $G\to\mathbb Z_2$, equivalently to subgroups of index $2$, each occurring exactly once, so $m_N=1$ for each kernel; there is a faithful action iff $G\cong\mathbb Z_2$.
    \end{proof}
    \begin{rem}
    Structural assumptions on $G$ (for instance a unique minimal normal subgroup) or on $GAS(G)$ (for instance connectedness together with a pendant vertex) do not force $GAS(G)=gas(G)$. For $G=\mathbb Z_9$ and $X=\{1,2,3\}$ the non-trivial actions are the two actions through the two $3$-cycles; both have kernel $3\mathbb Z_9$, so $GAS(\mathbb Z_9)=K_2$ is connected and every vertex is pendant, while $gas(\mathbb Z_9)=K_1$. By Theorem \ref{gas=GAS}, $GAS(G)=gas(G)$ holds only for $|X|=2$, whatever structural assumptions are made on $G$.
    \end{rem}

    We now check the significance of $gas(G)$ when $G$ is abelian. Let $\Gamma(G)$ denote the intersection graph of $G$ \cite{Akbari, Zelinka} whose vertices are the non-trivial proper subgroups of $G$, two vertices being adjacent if and only if their intersection is non-trivial.
    \begin{prop}\label{abelian gamma}
    Let $G$ be abelian. Then $gas(G)\cong\Gamma(G)$.
    \end{prop}
    \begin{proof}
    Every subgroup of $G$ is normal, and under the standing assumption every proper subgroup is a kernel. So the vertices of $gas(G)$ are the non-trivial proper subgroups, with adjacency by non-trivial intersection.
    \end{proof}
    For finite abelian groups, we obtain the following description of the groups for which $gas(G)$ has no $3$-cycle. (Note that $\mathbb Z_p\times\mathbb Z_p$ and $\mathbb Z_p$ occur in the list.)
    \begin{thm}\label{abelian triangle}
    Let $G\ne\{e\}$ be a finite abelian group. Then $gas(G)$ is $3$-cycle-free if and only if $G$ is isomorphic to $\mathbb Z_p$, $\mathbb Z_{p^2}$, $\mathbb Z_{p^3}$, $\mathbb Z_p\times\mathbb Z_p$ or $\mathbb Z_{p}\times \mathbb Z_q$ for some primes $p$ and $q$.
    \end{thm}
    \begin{proof}
    For the listed groups, $gas(G)$ is respectively empty, one vertex, two adjacent vertices, $p+1$ pairwise non-adjacent vertices and two non-adjacent vertices; none contains a triangle.
    Conversely, suppose $G$ is not in the list; we find three distinct non-trivial proper subgroups with a common non-trivial subgroup, which form a triangle.\\
    (1) If $|G|$ has three distinct prime divisors $p,q,r$, take elements $a,b,c$ of orders $p,q,r$; the subgroups $\langle ab\rangle,\langle ac\rangle,\langle bc\rangle$ have orders $pq,pr,qr$, are proper, and contain $\langle a\rangle$, $\langle b\rangle$ or $\langle c\rangle$ pairwise.\\
    (2) If $|G|=p^aq^b$ with $G=P\times Q$ and $|P|\ge p^2$, take $L<M$ in $P$ of orders $p$ and $p^2$; then $L$, $M$ and $L\times Q$ are proper non-trivial, distinct, and all contain $L$. The case $|Q|\ge q^2$ is symmetric; otherwise $G\cong\mathbb Z_p\times\mathbb Z_q$.\\
    (3) If $|G|=p^k$: for $k\ge4$ take a chain of subgroups of orders $p,p^2,p^3$. For $k=3$, if $G\cong\mathbb Z_{p^3}$ we are in the list; if $G\cong\mathbb Z_{p^2}\times\mathbb Z_p$, the subgroups $\langle(p,0)\rangle$, $\langle(1,0)\rangle$, $\langle(p,0),(0,1)\rangle$ form a triangle; if $G\cong\mathbb Z_p^3$, take two distinct subgroups of order $p^2$ and their (non-trivial) intersection. For $k\le2$, $G$ is in the list.
    \end{proof}
    Together with Corollary \ref{complete cor}(iii), $gas(G)$ is complete for a finite abelian group $G$ if and only if $G$ is a cyclic $p$-group.

    We have discussed earlier the behaviour of $GAS(D_n)$ for prime powers $n$. For odd $n$ we have the following.
    \begin{prop}\label{odd n}
    Let $n\ge3$ be odd. Then $gas(D_n)\cong\Gamma(\mathbb Z_n)$.
    \end{prop}
    \begin{proof}
    By Lemma \ref{conrad}, the proper normal subgroups of $D_n$ for odd $n$ are exactly the subgroups of $\langle r\rangle\cong\mathbb Z_n$; all occur as kernels under the standing assumption. Adjacency is non-trivial intersection.
    \end{proof}
    For even $n$ there are, in addition, the two vertices $\langle r^2,s\rangle$ and $\langle r^2,rs\rangle$, each adjacent exactly to the vertices $\langle r^d\rangle$ with $d$ even.

    We now discuss the significance of $gas(G)$ for cycles. The following result needs no solvability assumption.
    \begin{thm}\label{finite solvable n-cycle}
    Let $G$ be a finite group and $|X|\ge|G|$. If $GAS(G)$ has a cycle, then $girth(GAS(G))=3$.
    \end{thm}
    \begin{proof}
    Suppose the shortest cycle of $GAS(G)$ has length $k\ge4$. It is an induced cycle $\phi_1-\dots-\phi_k-\phi_1$; by the argument of Proposition \ref{transfer}(ii) its vertices have pairwise different non-trivial kernels $K_i$. Let $N=K_1\cap K_2$; it is a non-trivial proper normal subgroup, hence a kernel of some action $\psi$ (standing assumption).\\
    If $N\ne K_1$ and $N\ne K_2$, then $\psi$ is distinct from $\phi_1,\phi_2$, and $K_\psi\cap K_1=K_\psi\cap K_2=N\ne\{e\}$, so $\phi_1,\phi_2,\psi$ is a triangle, contradicting $k\ge4$ minimal.\\
    If $N=K_1$, then $K_1\subseteq K_2$, so $K_k\cap K_2\supseteq K_k\cap K_1\ne\{e\}$ and $\phi_k\sim\phi_2$; as $k\ge4$, this is a chord of the cycle, a contradiction. The case $N=K_2$ is symmetric (with $\phi_3$ and $\phi_1$).\\
    So the shortest cycle has length $3$.
    \end{proof}
    \begin{rem}
    Although the girth is $3$, $gas(G)$ may contain longer induced cycles, even for cyclic groups: in $gas(\mathbb Z_{210})$ the subgroups of orders $6,15,35,14$ form an induced $4$-cycle ($\gcd(6,15)=3$, $\gcd(15,35)=5$, $\gcd(35,14)=7$, $\gcd(14,6)=2$, and $\gcd(6,35)=\gcd(15,14)=1$).
    \end{rem}
    \begin{thm}\label{diameter}
    Let $G$ be a finite group, $|X|\ge|G|$, and assume that $gas(G)$ is connected with at least two vertices. Then $\operatorname{diam}(gas(G))\le2$, and so $\operatorname{diam}(GAS(G))\le 2$ whenever $GAS(G)$ is connected.
    \end{thm}
    \begin{proof}
    Let $K_1\ne K_2$ be non-adjacent vertices of $gas(G)$, and let $M_i\le K_i$ be minimal normal subgroups of $G$. Since $K_1\cap K_2=\{e\}$, $M_1\ne M_2$ and $M_1\cap M_2=\{e\}$. Suppose $M_1M_2=G$, so $G=M_1\times M_2$. If $N$ is a proper normal subgroup adjacent to $M_1$, then $N\cap M_1$ is a non-trivial normal subgroup of $G$ inside $M_1$, so $N\supseteq M_1$ and $N=M_1(N\cap M_2)$ is $M_1$ or $G$; so $N=M_1$. Hence $M_1$ is an isolated vertex of $gas(G)$, contradicting connectedness. Therefore $M_1M_2$ is a proper normal subgroup. It is different from $K_1$ and $K_2$ (otherwise $M_2\le K_1\cap K_2$ or $M_1\le K_1\cap K_2$), and it meets $K_1$ in a group containing $M_1$ and $K_2$ in a group containing $M_2$. So $K_1-M_1M_2-K_2$ is a path of length $2$. The statement for $GAS(G)$ follows from Proposition \ref{transfer}(i).
    \end{proof}

    The following result determines the independence number.
    \begin{prop}\label{independence}
    Let $G$ be a finite group which is not simple, $|X|\ge|G|$, and let $\mu(G)$ be the number of minimal normal subgroups of $G$. Then the independence number of $gas(G)$ equals $\mu(G)$, and the independence number of $GAS(G)$ equals $\mu(G)+m_{\{e\}}$, where $m_{\{e\}}$ is the number of faithful actions. If $G$ is simple, then $gas(G)$ is empty.
    \end{prop}
    \begin{proof}
    As $G$ is not simple, every minimal normal subgroup is proper, hence a vertex. Distinct minimal normal subgroups intersect trivially, so they form an independent set. Conversely let $N_1,\dots,N_s$ be an independent set, and let $M_i\le N_i$ be minimal normal. If $M_i=M_j$ for $i\ne j$, then $N_i\cap N_j\ne\{e\}$; so $M_1,\dots,M_s$ are distinct and $s\le\mu(G)$. (It is not enough to observe that every normal subgroup contains \emph{some} minimal normal subgroup: this shows only that the set of minimal normal subgroups is a maximal independent set.) In $GAS(G)$ an independent set contains at most one vertex from each clique of twins and may contain all isolated (faithful) vertices.
    \end{proof}

    We now use the characterization of derived graphs by Beineke \cite{Beineke}. We recall that a graph $\Gamma$ is called a \emph{line graph} (derived graph) if it is the derived graph of some graph $H$, written $\Gamma=L(H)$.
        \begin{rem}\cite{Beineke}\label{Beineke}
            $\Gamma$ is a derived graph of some graph if and only if none of the nine graphs defined by Beineke in \cite{Beineke} is an induced subgraph of $\Gamma$.
        \end{rem}
    Two of the nine graphs will be enough for us: the claw $K_{1,3}$, and $K_5-e$ (the complete graph on five vertices with one edge removed). That they are not derived graphs can also be seen directly. If $K_{1,3}=L(H)$, the three leaves are three pairwise non-incident edges of $H$ all incident to the edge corresponding to the centre, which has only two endpoints, a contradiction. If $K_5-e=L(H)$, the four vertices of the $K_4$ are four pairwise incident edges of $H$, hence four edges at one vertex $c$ (a triangle supports only three edges); the fifth vertex $y$ is adjacent to exactly three of them, so the edge $y$ does not contain $c$ and meets three edges with three different second endpoints, but an edge has only two endpoints, a contradiction.
        We also mention a remark in the form that follows from the complement.
        \begin{rem}\label{Remark Mathil}
            A graph $\Gamma$ is the complement of a derived graph if and only if $\Gamma$ contains none of the complements of the nine graphs of Beineke \cite{Beineke} as an induced subgraph (see also \cite{Mathil}). In particular, if $\Gamma$ contains an induced $K_3\cup K_1$ (the complement of the claw) or an induced $K_2\cup 3K_1$ (the complement of $K_5-e$), then $\Gamma$ is not the complement of a derived graph.
        \end{rem}

    Consider a finite nilpotent group $G$ of order $\prod\limits_{i=1}^kp_i^{\alpha_i}$, where $p_i$'s are distinct prime numbers, so that $G=P_1\times\dots\times P_k$ with $P_i$ the Sylow $p_i$-subgroup. It has been observed that Conrad in \cite{Conrad 2} gives the following result on $G$.
\begin{rem}\cite{Conrad 2}\label{Conrad2}
	For every divisor $d$ of $|G|$, there is a normal subgroup of $G$ of order $d$.
\end{rem}
Throughout this part, $k\ge2$, $|X|\ge|G|$, and $gas(G)$ is the intersection graph of the non-trivial proper normal subgroups of $G$; in drawings we do not include the vertex $G$ (it is not a vertex). We use the following standard facts. For a finite $p$-group $P$ of order $p^n$:
\begin{enumerate}
\item[(F1)] every non-trivial normal subgroup of $P$ meets $Z(P)$ non-trivially; hence it contains a minimal normal subgroup, which is a subgroup of order $p$ of $Z(P)$;
\item[(F2)] $P$ has a unique minimal normal subgroup if and only if $Z(P)$ is cyclic; otherwise $P$ has at least $p+1\ge3$ minimal normal subgroups;
\item[(F3)] $P$ has a normal subgroup of order $p^j$ for each $0\le j\le n$ (Remark \ref{Conrad2});
\item[(F4)] if $P$ is not cyclic, $P$ has at least $p+1\ge3$ maximal subgroups, all normal.
\end{enumerate}
For a nilpotent $G=P_1\times\dots\times P_k$, the normal subgroups of $G$ are exactly the products $N_1\times\dots\times N_k$ with $N_i\trianglelefteq P_i$, the intersection is computed componentwise, and two normal subgroups $N=\prod N_i$ and $M=\prod M_i$ are adjacent in $gas(G)$ if and only if $N_i\cap M_i\ne\{e\}$ for some $i$. The minimal normal subgroups of $G$ are the minimal normal subgroups of the $P_i$.
\begin{lemma}\label{noncyclic nu}
Let $P$ be a $p$-group with a unique minimal normal subgroup, and let $\nu(P)$ be the number of non-trivial normal subgroups of $P$ (including $P$). If $P$ is cyclic of order $p^n$, then $\nu(P)=n$. If $P$ is not cyclic, then $\nu(P)\ge5$.
\end{lemma}
\begin{proof}
The cyclic case is clear. Let $P$ be non-cyclic, of order $p^n$, with unique minimal normal subgroup $L$. If $n=2$, then $P\cong\mathbb Z_p^2$ has $p+1$ minimal normal subgroups, so $n\ge3$. By (F4), $P$ has at least three maximal subgroups, all normal, of order $p^{n-1}\ge p^2$, hence different from $L$ and from $P$. Together with $L$ and $P$ this gives $\nu(P)\ge5$.
\end{proof}
We now determine for which finite nilpotent groups $gas(G)$ is a derived graph.
\begin{thm}\label{nilpotent}
Let $G$ be a finite nilpotent group of order divisible by at least two primes. Then $gas(G)$ is a derived graph if and only if $G$ is cyclic of order $p_1p_2$, $p_1p_2p_3$, $p_1^2p_2$ or $p_1^3p_2$, where $p_1,p_2,p_3$ are distinct primes.
\end{thm}
\begin{proof}
($\Leftarrow$) We give a graph $H$ with $gas(G)=L(H)$ in each case. For $\mathbb Z_{p_1p_2}$ the graph $gas(G)$ is two isolated vertices, which is $L(2K_2)$. For $\mathbb Z_{p_1^2p_2}$ the vertices are the subgroups of orders $p_1,p_1^2,p_2,p_1p_2$; $gas(G)$ is a triangle $\{p_1,p_1^2,p_1p_2\}$ with the pendant vertex $p_2$ attached to $p_1p_2$ (the paw), and it is $L(H)$ where $H$ is the star $K_{1,3}$ with one extra edge attached at one leaf. For $\mathbb Z_{p_1^3p_2}$ the subgroups of orders $p_1,p_1^2,p_1^3,p_1p_2,p_1^2p_2$ form a $K_5$, and $p_2$ is adjacent exactly to $p_1p_2$ and $p_1^2p_2$; this is $L(H)$ with $H$ a star $K_{1,5}$ with centre $c$ and leaves $u_1,\dots,u_5$ together with the edge $u_4u_5$. For $\mathbb Z_{p_1p_2p_3}$ the vertices $p_1p_2,p_1p_3,p_2p_3$ form a triangle and each $p_i$ is adjacent exactly to the two of them containing $p_i$ (the $3$-sun); this is $L(H)$ with $H$ the net (a triangle with a pendant edge at each vertex).\\
($\Rightarrow$) Let $G=P_1\times\dots\times P_k$ be nilpotent with $k\ge2$ and let $L_i\le P_i$ be a minimal normal subgroup (of order $p_i$).\\
\emph{Case $k\ge3$.} If $|G|\ne p_1p_2p_3$, then $C=L_1L_2L_3$ is a proper non-trivial normal subgroup, adjacent to $L_1,L_2,L_3$, which are pairwise non-adjacent. This is an induced claw, so $gas(G)$ is not derived. If $|G|=p_1p_2p_3$ then $G$ is cyclic and we are done.\\
\emph{Case $k=2$.} Write $G=P\times Q$. \\
(a) Suppose one of the Sylow subgroups, say $P$, has more than one minimal normal subgroup; by (F2) it has at least three, $L_1,L_2,L_3$. Let $M$ be a minimal normal subgroup of $Q$. If $P\ne L_1L_2$ or $Q\ne M$, the subgroup $C=L_1L_2\times M$ is a proper normal subgroup adjacent to $L_1\times\{e\},L_2\times\{e\},\{e\}\times M$, which are pairwise non-adjacent: an induced claw. If $P=L_1L_2$ and $Q=M$, then $P$ has order $p^2$, and $P\times\{e\}$ is adjacent to $L_1\times\{e\},L_2\times\{e\},L_3\times\{e\}$, again an induced claw.\\
(b) Suppose both $P$ and $Q$ have a unique minimal normal subgroup, $L$ and $M$ respectively. Let $a=\nu(P)$ and $b=\nu(Q)$ be the numbers of non-trivial normal subgroups. Every vertex $N\times R$ of $gas(G)$ contains $L\times\{e\}$ or $\{e\}\times M$ (property (F1)). Suppose $ab\ge4$. Then there are at least three distinct pairs $(N_i,R_i)$ of non-trivial normal subgroups of $P$ and $Q$ other than $(P,Q)$. Put $x_i=N_i\times R_i$ ($i=1,2,3$), $x_4=\{e\}\times Q$ and $y=P\times\{e\}$. Then $x_1,x_2,x_3$ contain $L\times\{e\}$ and $\{e\}\times M$, so $x_1,\dots,x_4$ are pairwise adjacent; $y$ is adjacent to $x_1,x_2,x_3$ (they contain $L\times\{e\}$), but $y\cap x_4=\{e\}$. Hence $\{x_1,\dots,x_4,y\}$ induces $K_5-e$ and $gas(G)$ is not derived. So $ab\le3$, that is, $\{a,b\}=\{1,1\},\{1,2\}$ or $\{1,3\}$, and in particular $a,b\le3<5$. By Lemma \ref{noncyclic nu}, $P$ and $Q$ are cyclic, of orders $p^a$ and $q^b$. Thus $G$ is cyclic of order $pq$, $p^2q$ (or $pq^2$) or $p^3q$ (or $pq^3$).
\end{proof}
\begin{rem}
Non-cyclic Sylow subgroups never give a derived graph. For instance, for $G=D_4\times\mathbb Z_3$ (order $24$) or $Q_8\times\mathbb Z_3$, $gas(G)$ contains the induced $K_5-e$ of case (b), although all normal subgroups of the Sylow $2$-subgroup contain its centre; the vertex $P\times\{e\}$ is adjacent only to the vertices $N\times R$ with $N\ne\{e\}$. Similarly, a non-cyclic Sylow $p_1$-subgroup of order $p_1^2$ gives an induced claw by case (a). The theorem above has been checked by computer on all nilpotent groups of order at most $3000$ whose prime divisors lie in $\{2,3,5,7\}$ and whose Sylow $2$-, $3$-, $5$-, $7$-subgroups are of order at most $8$, $27$, $25$, $7$, respectively (this includes all abelian groups, $D_4$, $Q_8$ and the two non-abelian groups of order $27$), and on cyclic groups of order up to $2\cdot10^5$ with prime divisors in $\{2,3,5,7\}$.
\end{rem}

We now explore the conditions under which the complement of $gas(G)$ is a derived graph.
\begin{thm}\label{nilpotent second}
Let $G$ be a finite nilpotent group of order divisible by at least two primes. Then the complement of $gas(G)$ is a derived graph if and only if $G$ is cyclic of order $p_1p_2$, $p_1p_2p_3$, $p_1^2p_2$ or $p_1^2p_2^2$, where $p_1,p_2,p_3$ are distinct primes.
\end{thm}
\begin{proof}
($\Leftarrow$) The complement of $gas(\mathbb Z_{p_1p_2})$ is $K_2=L(P_2)$ (here $P_2$ is a path with two edges). The complement of the paw $gas(\mathbb Z_{p_1^2p_2})$ is a path on three vertices (joining $p_2$ to $p_1$ and $p_1^2$) together with the isolated vertex $p_1p_2$, which is $L$ of a path with three edges together with a disjoint edge. In $gas(\mathbb Z_{p_1^2p_2^2})$, two vertices are non-adjacent exactly when one is a subgroup of order $p_1$ or $p_1^2$ and the other of order $p_2$ or $p_2^2$; so the complement is a $4$-cycle on these four subgroups together with three isolated vertices, which is $L(C_4\cup3K_2)$. In $gas(\mathbb Z_{p_1p_2p_3})$ the complement is a triangle $\{p_1,p_2,p_3\}$ with the pendant vertices $p_2p_3,p_1p_3,p_1p_2$ attached to $p_1,p_2,p_3$ respectively, that is, the net, which is the line graph of the graph obtained from a claw by subdividing each edge once.\\
($\Rightarrow$) We use that, by Remark \ref{Remark Mathil}, an induced $K_3\cup K_1$ in $gas(G)$ shows that the complement is not derived. Write $G=P_1\times\dots\times P_k$ and let $L_i\le P_i$ be minimal normal.\\
\emph{Case $k\ge3$.} Suppose $|G|\ne p_1p_2p_3$ (otherwise $G$ is cyclic and we are done). If some Sylow subgroup, say $P_1$, has order at least $p_1^2$, choose by (F3) a normal subgroup $M'$ of $P_1$ of order $p_1^2$ and, by (F1), a minimal normal subgroup $M\le M'$. Then $M$, $M'$ and $M\times L_2$ are three distinct, pairwise adjacent vertices (each contains $M$), and $L_3$ is non-adjacent to all of them: an induced $K_3\cup K_1$. Otherwise all Sylow subgroups have prime order, so $k\ge4$ and $G$ is cyclic; then $L_1$, $L_1L_2$, $L_1L_3$ are pairwise adjacent and $L_4$ is non-adjacent to all of them.\\
\emph{Case $k=2$.} Write $G=P\times Q$.\\
(a) If $P$ has several minimal normal subgroups, it has at least three, $L_1,L_2,L_3$ by (F2), and $\{e\}\times Q$, $L_2\times Q$, $L_3\times Q$ are three distinct proper vertices which are pairwise adjacent (each contains $\{e\}\times Q$), while $L_1\times\{e\}$ is non-adjacent to all of them: an induced $K_3\cup K_1$. The same holds if $Q$ has several minimal normal subgroups.\\
(b) Suppose $P$ and $Q$ have a unique minimal normal subgroup each, and let $a=\nu(P)$, $b=\nu(Q)$. If $b\ge3$, take three distinct non-trivial normal subgroups $R_1,R_2,R_3$ of $Q$; the vertices $\{e\}\times R_j$ are pairwise adjacent, and $P\times\{e\}$ is non-adjacent to all of them. Likewise if $a\ge3$. Hence $a,b\le2<5$, and by Lemma \ref{noncyclic nu}, $P$ and $Q$ are cyclic, of orders $p$ or $p^2$ and $q$ or $q^2$, respectively. So $G$ is cyclic of order $pq$, $p^2q$, $pq^2$ or $p^2q^2$.
\end{proof}
\begin{rem}
In particular, the order $4p_2$ with non-cyclic Sylow $2$-subgroup does not occur: in $gas(\mathbb Z_2^2\times\mathbb Z_3)$ the subgroups $\{e\}\times\mathbb Z_3$, $\langle b\rangle\times\mathbb Z_3$, $\langle c\rangle\times\mathbb Z_3$ form a triangle and $\langle a\rangle\times\{e\}$ is non-adjacent to all of them (case (a) of the proof), so the complement contains a claw. Theorems \ref{nilpotent} and \ref{nilpotent second} have been checked by computer as described above.
\end{rem}

	\section*{Statements and Declarations}
	\textbf{Competing interests.} The authors have no relevant financial or non-financial interests to disclose.\\[4pt]
	\textbf{Funding.} The authors declare that no funds, grants, or other support were received during the preparation of this manuscript.\\[4pt]
	\textbf{Data availability.} Data sharing is not applicable to this article, as no datasets were generated or analysed.\\[4pt]
	\textbf{Author contributions.} All authors contributed to the study conception and design, and all authors read and approved the final manuscript.

\end{document}